\documentclass[leqno,11pt]{amsart}
\usepackage{amssymb, amsmath}

\usepackage{xcolor,ulem,textcomp,MnSymbol}
\usepackage[pageref]{backref}
\usepackage{url}% for url's in bib
\usepackage[mathscr]{euscript}
\renewcommand*{\backref}[1]{}

\numberwithin{equation}{section}
\usepackage{enumerate}
\usepackage{graphicx}
\usepackage{cancel}

\usepackage{hyperref}

\theoremstyle{definition}
\theoremstyle{plain}
\newtheorem{thm}{Theorem}[section]%[chapter]%[section]%!!the counter really doesn't follow the chapters!!
\newtheorem{Prop}[thm]{Proposition}

 \newtheorem{Thm}{Theorem}[section]
 \newtheorem{Rmk}[thm]{Remark}
 \newtheorem{Lem}[thm]{Lemma}
  
\def\bbE {\mathbb {E}}

\def\bbH {\mathbb{H}}
\def\bbP {\mathbb {P}}

\def\R {\mathbb{R}}
\def\bbR {\mathbb{R}}

\def\T {\mathbb{T}}
\def\bbW {\mathbb{W}}
\def\Z {\mathbb{Z}}

  \definecolor{darkgreen}{rgb}{0.0, 0.5, 0.0}

\def\cA {\mathcal{A}}

\def\cD {\mathcal{D}}

\def\cG {\mathcal{G}}

\def\cL {\mathcal{L}}

\def\cM {\mathcal{M}}
\def\cN {\mathcal{N}}

\def\cZ {\mathcal{Z}}

\def\eps {{\varepsilon}}
\def\e {{\varepsilon}}

\def\indc {{\bf 1}}

\def\la {\langle}
\def\ra {\rangle}

\def\d {{\partial}}

\newcommand{\Cov}{\operatorname{Cov}}

\newcommand{\Span}{\operatorname{span}}

\newcommand{\Ker}{\operatorname{Ker}}

\newcommand{\ba}{\begin{aligned}}
\newcommand{\ea}{\end{aligned}}

\newcommand{\be}{\begin{equation}}
\newcommand{\ee}{\end{equation}}

\numberwithin{equation}{section}
\begin{document}%%%%%%%%%%%%%%%%%%%%%%%%%%%%%%%%
%%%%%%%%%%%%%%%%%%%%%%%%%%%%%%%%%%%%%%%%%%%%%%%%

% Vous pouvez mettre dans la prochain ligne la rubrique choisie
% (si vous la connaissez) - meme deux, format : Rubrique1/Rubrique2
\title[Fluctuating hydrodynamics for dilute gases] {Dynamics of dilute gases  at  equilibrium:\\
 from the atomistic description \\
 to fluctuating hydrodynamics}
\author{Thierry Bodineau, Isabelle Gallagher, Laure Saint-Raymond, Sergio Simonella}
% utiliser les \`etiquettes pour indiquer l'adresse de chaque auteur,
%     s'il y a plusieurs adresses
\begin{abstract} 
  We   derive linear fluctuating hydrodynamics as the low density limit of a deterministic system of particles at equilibrium. 
The proof builds upon the results of~\cite{BGSS3} where the asymptotics of the covariance of the fluctuation field is obtained, and on the proof of the Wick rule for the fluctuation field in~\cite{BGSS4}.
\end{abstract}

\maketitle

This  paper is dedicated to the memory of K. Gawedzki. It shows how noise in hydrodynamic models of perfect gases can emerge from a deterministic microscopic dynamics. It is reminiscent of the concept of spontaneous stochasticity introduced in \cite{BGK} and  formalized  in  \cite{CGHKV}.

\section{The different levels of modeling}

\subsection{The atomistic description}
The microscopic model consists of  identical hard spheres of unit mass and of diameter~$\eps$.
The  motion of $N$ such hard spheres is ruled by a system of ordinary differential equations, which are set in~$ ( \T ^d\times\R^d)^{N }$ where~$\mathbb T^d$ is  the unit~$d$-dimensional periodic box with $d \geq 3$: writing~${\bf x}^{\e}_i \in  \T ^d$ for the position of the center of the particle labeled by~$i$ and~${\bf v}^{\e}_i \in  \R ^d$ for its velocity, one has
\begin{equation}
\label{hardspheres}
{d{\bf x}^{\e}_i\over dt} =  {\bf v}^{\e}_i\,,\quad {d{\bf v}^{\e}_i\over dt} =0 \quad \hbox{ as long as \ } |{\bf x}^{\e}_i(t)-{\bf x}^{\e}_j(t)|>\eps  
\quad \hbox{for \ } 1 \leq i \neq j \leq N
\, ,
\end{equation}
with specular reflection at collisions: 
\begin{equation}
\label{defZ'nij}
\begin{aligned}
\left. \begin{aligned}
 \left({\bf v}^{\e}_i\right)'& := {\bf v}^{\e}_i - \frac1{\eps^2} ({\bf v}^{\e}_i-{\bf v}^{\e}_j)\cdot ({\bf x}^{\e}_i-{\bf x}^{\e}_j) \, ({\bf x}^{\e}_i-{\bf x}^{\e}_j)   \\
\left({\bf v}^{\e}_j\right)'& := {\bf v}^{\e}_j + \frac1{\eps^2} ({\bf v}^{\e}_i-{\bf v}^{\e}_j)\cdot ({\bf x}^{\e}_i-{\bf x}^{\e}_j) \, ({\bf x}^{\e}_i-{\bf x}^{\e}_j)  
\end{aligned}\right\} 
\quad  \hbox{ if } |{\bf x}^{\e}_i(t)-{\bf x}^{\e}_j(t)|=\eps\,.
\end{aligned}
\end{equation}
 This flow does not cover all possible situations, as multiple collisions are excluded.
But one can show (see \cite{Ale75}) that for almost every admissible initial configuration~$({\bf x}^{\e 0}_i, {\bf v}^{\e 0}_i)_{1\leq i \leq N}$, there are neither multiple collisions,
nor accumulations of collision times, so that the dynamics is globally well defined. 

We will not be interested here in one specific realization of this deterministic dynamics,  but rather in a statistical description. This is achieved by introducing a measure at time~0, on the phase space we now specify.
The collections of~$N$ positions and velocities are denoted respectively by~$X_N := (x_1,\dots,x_N) $ in~$ \T^{dN}$ and~$V_N := (v_1,\dots,v_N) $ in~$ \R^{d N}$,  and we set~$Z_N:= (X_N,V_N) $, with~$Z_N =(z_1,\dots,z_N)$, $z_i = (x_i,v_i)$. A set of~$N$ particles is characterized by a random variable ${\mathbf Z}^{\eps 0}_N =  ({\mathbf z}^{\eps 0}_1,\dots , {\mathbf z}^{\eps 0}_N)$ specifying the time-zero configuration in the phase space 
\begin{equation}
\label{D-def}
{\mathcal D}^{\eps}_{N} := \big\{Z_N \in (\T ^d\times\R^d)^N \, / \,  \forall i \neq j \, ,\quad |x_i - x_j| > \eps \big\} \, ,
\end{equation}
and an evolution
$$ t \longmapsto {\mathbf Z}^{\eps}_N(t) = \big({\mathbf z}^{\eps}_1(t),\dots , {\mathbf z}^{\eps}_N(t)\big)\;,\qquad t>0$$
according to the deterministic flow \eqref{hardspheres}-\eqref{defZ'nij} (well defined with probability 1).

To avoid spurious correlations due to a given total number of particles, we   actually consider a grand canonical state (as in \cite{Ki75,BLLS80}), set on the phase space 
$$
{\mathcal D}^{\eps} := \bigcup_{N \geq 0}{\mathcal D}^{\eps}_{N}
$$
(notice that ${\mathcal D}^{\eps}_{N} = \emptyset$ for $N$ large). This means that the total number of particles is also a random variable,  which we shall denote by $\cN$.

More precisely, at equilibrium the probability density of finding $N$ particles at configuration~$Z_N$ is given by
\begin{equation}
\label{eq: initial measure}
\frac{1}{N!}W^{\eps }_{N}(Z_N) 
:= \frac{1}{\cZ^ \eps} \,\frac{\mu_\eps^N}{N!} \, \indc_{
{\mathcal D}^{\eps}_{N}}(Z_N) 
 \, \cM^{\otimes N}  (V_N)\, , \qquad \hbox{ for } N= 0, 1,2,\dots \end{equation} 
 for some (large)~$\mu_\eps$ to be fixed   below, with 
 \begin{equation*}
\label{eq: max}
  \cM(v) := \frac1{(2\pi)^\frac d2} \exp\big( {-\frac{ |v|^2}2}\big)\,,
\qquad \cM^{\otimes N}(V_N)= \prod_{i=1}^N \cM(v_i)\, , 
\end{equation*}
and the partition function is given by
\begin{equation*}
\label{eq: partition function}
\cZ^\eps :=  1 + \sum_{N\geq 1}\frac{\mu_\eps^N}{N!}  
\int_{\T^{dN} }   \prod_{i\neq j}\indc_{ |x_i - x_j| > \eps} \,  dX_N\,.
\end{equation*}
Here and below, $\indc_A$ will be the characteristic function of the set $A$.
The probability of an event~$A$ with respect to the equilibrium measure~(\ref{eq: initial measure}) will be denoted~${\mathbb P}_\eps(A)$, and~${\mathbb E}_\eps$ will be the expected value.
Definition \eqref{eq: initial measure} ensures that $$\mu_\e^{-1}{\mathbb E}_\eps\left(\cN\right)\to 1$$ as~$\mu_\e\to \infty$ with $\mu_\e \e^d \ll 1$.

\subsection{The kinetic description}
%
%In the low density regime, referred to as  the Boltzmann-Grad scaling, the density (average $\cN$) is tuned by the parameter 
%\begin{equation}
%\label{defmu}
%\mu_\e := \alpha^{-1} \e^{-(d-1)}\;,
%\end{equation} 
% ensuring that the mean free path between collisions is of order $\alpha$ (see~\cite{Gr49}). 
Let us define
the empirical measure  of the hard-sphere model
\begin{equation}
\label{eq: empirical}
\pi^{\eps}_t :=\frac{1}{\mu_\e}\sum_{i=1}^\cN \delta_{{\bf z}^\e_i(t)}  \, .
\end{equation} 
Under the invariant measure~(\ref{eq: initial measure}), it is not hard to see that if $\mu_\eps  \eps^d \to  0$ then~$\pi^{\eps}_t $  concentrates on~$\cM$:
for any test function $h: \T ^d\times\R^d\to\R$ and any $\delta>0$, $t \in \R$,
\begin{equation}\label{LLN}
\bbP_\eps \left( \Big|  \pi^\eps_t(h)-  
   \int_{\T ^d\times\R^d}  \cM (v)h(z)\Big| > \delta \right) \xrightarrow[\mu_\eps \to \infty]{} 0\;,
\end{equation}
which can be interpreted as a law of large numbers.

\smallskip

The fluctuations of the empirical density $\pi^{\eps}_t$ around its equilibrium value are described by the fluctuation field $\zeta^\eps_t$ defined by
\begin{equation}
\label{eq: fluctuation field}
\ \zeta^\eps_t(h):=  { \sqrt{\mu_\eps }} \, 
\Big(   \pi^\eps_t(h) -  {\mathbb E}_\eps\big(  \pi^\eps_t(h) \big)     \Big)\;,
 \end{equation}
 for any test function $h$.  
Initially~$\zeta^\eps_0$ converges in law towards a Gaussian white noise~$\zeta_0$ with covariance 
\begin{equation}
\label{eq: invariant measure}
\bbE \big( \zeta_0(h_1)\, \zeta_0(h_2) \big)
= \int h_1(z) \,h_2(z)\,\cM (v) \,dz\, . 
\end{equation}
As the measure is invariant, this covariance is constant  in time.
Let us define the mean free path 
$$\alpha:= (\mu_\eps \eps^{d-1})^{-1}\, ,  
$$
 and assume that~$
\alpha^{-1}  \geq 1 $ is bounded or slowly diverging, corresponding to the low density scaling. In this scaling it  has been proved in \cite{BGSS3,BGSS4} that~$ (\zeta^\eps_t)_{[0,T]}$ converges in law for all times~$T$ to a weak solution of  the fluctuating Boltzmann equation 
\begin{equation}
\label{eq: OU}
d \zeta_t   = \left(  - v \cdot \nabla_x  -\frac1\alpha \cL\right)  \,\zeta_t\, dt + d\eta_t\,,
\end{equation}
where the linearized collision operator is given by
\begin{equation} \label{eq:LBCO}
   \cL g :=  \int_{\R^d \times {\mathbb S}^{d-1}}  \cM(w) \big( (v - v_*) \cdot \omega \big)_+ \left[g(v) +  g(v_*) - g(v') -  g(v_*') \right]  d v_*\,d \omega
 \end{equation}
 with notation
\begin{equation}
\label{scattlaw}
v' = v- ((v - v_*) \cdot \omega)\,  \omega\, , \quad v_*' = v_*+  ((v-v_*) \cdot \omega)\,  \omega
\end{equation}
for the precollisional velocities obtained upon scattering, 
and  $d \eta_t(x,v)$ is a stationary Gaussian noise, explicitly characterized (see \cite{S83}). It has zero mean and covariance
\begin{equation}
\label{eq: noise covariance}
\begin{aligned}
& \bbE \left( \int_0^T   \! \! dt   \! \!  \int dz   h^{(1)} (z)  \eta_{t} (z)     \int_0^T    \! \!  dt_*   \! \! \int dz_* \, h^{(2)} (z_*) \eta_{t_*} (z_*) \right)
 \\ &\qquad \qquad= \frac{1}{2\alpha} \int_0^T    \! \!  dt  \! \!  \int d\mu(z, z_*, \omega) 
 \cM(v)\, \cM(v_*) \Delta h^{(1)}  \, \Delta h^{(2)} 
\end{aligned}
\end{equation}
denoting
$$
d\mu (z, z_*, \omega): = \delta_{x - x_*}  \,  \big( (v - v_*) \cdot \omega \big)_+ d \omega\, d v\, d v_* dx 
$$
and defining  for any $h$
$$  \Delta h (z, z_*, \omega) := h(z') +  h(z_*') -  h(z) -  h(z_*)\,, $$
where~$z_i':=(x_i,v_i')$ with notation
(\ref{scattlaw}) 
for the velocities obtained upon scattering. Note that this noise is white in time and space, but correlated in  
velocities.

\subsection{The hydrodynamic description}

It is by now classical (see \cite{BGL2, GL, BGSR2} and references therein)  that the solutions to the  scaled linearized Boltzmann equation
\begin{equation} 
\label{eq:scaled- lbeq}
 \partial_t g_\alpha +v\cdot \nabla_x g_\alpha  +\frac1\alpha \cL g_\alpha =0\,, \quad g_\alpha (0) = g^0
\end{equation} converge in the fast relaxation limit $\alpha \to 0$ towards the local thermodynamic equilibrium
$$ g(t,x,v) = \rho(t,x) +u(t,x) \cdot v +\theta(t,x) {|v|^2- d \over 2} $$
where $\rho, u, \theta$ satisfy the acoustic equations
\begin{equation}
\label{acoustic}
\begin{cases}
\d_t  \rho +  \nabla_x \cdot  u = 0\,\\
\d_t u + \nabla_x (\rho +\theta) = 0 \\
 \d_t \theta  + \frac 2 d\nabla_x \cdot u  = 0  
\end{cases}
\end{equation}
and the initial data is the projection of~$g^0$ onto hydrodynamic modes
$$
\begin{aligned}
& \rho_{|t=0}(x) :=  \int g_0 (x,v)    \cM(v) \, dv 
\, , \quad u_{|t=0}(x) :=   \int g_0 (x,v) v \cM(v)  \, dv \, , \\ 
& \theta_{|t=0}(x) :=  \int g_0 (x,v)    \big(\frac{|v|^2 }d - 1\big)  \cM(v) \, dv \, .
\end{aligned}
$$
  In the linearized equation \eqref{eq:scaled- lbeq}, the frequency of collisions 
$1/\alpha$ has been tuned according to the hyperbolic scaling. 
The diffusive regime can then be found by  rescaling time by a factor~$1/\alpha$. 
In this way, one can also obtain the weak convergence (which actually filters out the fast oscillating  acoustic waves) 
\begin{equation}
\label{defutheta}
 g_\alpha \left(\frac \tau\alpha, x, v \right) \rightharpoonup u(\tau,x) \cdot v +\theta(\tau,x)  {|v|^2 -(d+2) \over 2} 
 \end{equation}
towards diffusive fluid models, namely the incompressible Stokes-Fourier equations
\begin{equation}
\label{Stokes}
\begin{cases}
\d_\tau  u = \nu  \Delta_x u  , \qquad \nabla_x \cdot  u = 0\\
 \d_\tau \theta  = \kappa  \Delta_x \theta   \,,
\end{cases}
\end{equation}
where the diffusion coefficients $\nu$ and $\kappa$ depend only on the linearized collision operator $\cL$ (they are defined explicitly in (\ref{mu-kappa}) below). The initial data is the projection of~$g^0$ onto  non oscillating hydrodynamic  modes
\begin{equation}
\label{eq: initial data}
u_{|\tau=0}(x) := P \int g_0v \cM(v)  \, dv \, , \quad \theta_{|\tau=0}(x) :=  \int g_0   \big( \frac{|v|^2}{ d+2} - 1\big)   \cM(v) \, dv 
\end{equation}
where~$ P$ is  the Leray projection on divergence free vector fields.
In the following, we refer to non oscillating modes as those satisfying the 
incompressibility and Boussinesq constraints (see \eqref{eq: non oscillating}).

\subsection{Fluctuating hydrodynamics}
\label{sec: Fluctuating hydrodynamics}

In the hyperbolic regime corresponding to \eqref{acoustic}, the fluctuation-dissipation principle predicts that there will be no dynamical fluctuation and the  fluctuation field tested against hydrodynamical modes $(\rho,u,\theta)$ is simply transported by the acoustic equation. In contrast, in the diffusive regime, when taking into account the noise at kinetic level (i.e.\;starting with~(\ref{eq: OU})), we expect to obtain  fluctuating hydrodynamics.
In the following, we will focus on this more interesting case.
We refer to~\cite{S2}, Section 7.1 for the general theory of hydrodynamic fluctuations, which was first developed for equilibrium states in \cite{LL6}. The link with the predictions from kinetic theory in the case of dilute gases was discussed in \cite{KCD82} (see also \cite{BB} for a recent contribution).
%We refer the reader to Chapter 7 of \cite{S2} for an extensive discussion on these aspects.

Let us 
define a  joint process by time rescaling and projecting on non oscillating hydrodynamic modes
the  fluctuation field~$\zeta^\eps_t$ defined in~(\ref{eq: fluctuation field}). According to~(\ref{defutheta}) we   consider, for any pair of test functions~$ (\varphi,\psi)\in C^\infty(\T^d; \R^{d} \times \R) $ with~$ \nabla_x \cdot \varphi  = 0$, the fluctuation field
$$
  \zeta^\eps_{t} \big( \varphi  \cdot v \big)+ \zeta^\eps_{t } \Big( \psi \big(  {|v|^2 \over d+2}-1\big)\Big)\, .
$$
 To simplify the notation, we denote  from now  on the couple of test functions   by
 \begin{equation}
\label{def test function}
\phi = (\varphi,\psi)\in C^\infty(\T^d; \R^{d} \times \R) \, , \,  \quad\nabla_x \cdot \varphi  = 0
\end{equation}
and to recover a  diffusive regime,   time is rescaled as follows:
\begin{equation}
\label{fluct-u-theta}
\begin{aligned}
  \xi^\eps_\tau(\phi) 
&:=  {\mathcal   U}^\eps_\tau(\varphi) +    {\Theta}^\eps_\tau(\psi) 
\\
&:=  \zeta^\eps_{\tau/\alpha} \big( \varphi  \cdot v \big)+ \zeta^\eps_{\tau/\alpha} \Big( \psi \big(  {|v|^2 \over d+2}-1\big)\Big) \,.\end{aligned}
\end{equation}
  We stress the fact  that in contrast with~$ \zeta^\eps$, the test functions in~$\xi^\eps$ only depend on the space variable.
In the limit $\mu_\eps \to \infty$ with $\alpha$ slowly vanishing, we expect the fluctuation fields~$\big(  {\mathcal U}^\eps ,   { \Theta}^\eps  \big)$ to converge in the sense of distributions to~$( {\mathcal U} ,\Theta)$ solving   the fluctuating Stokes-Fourier equations
\begin{equation}
\label{fluct-Stokes}
\begin{cases}
\d_\tau  {\mathcal U} = \nu  \Delta_x   {\mathcal U}  + \sqrt{2 \nu} \, P \, \nabla \cdot  \dot \bbW_t , \\
 \d_\tau \Theta  = \kappa  \Delta_x \Theta  + \sqrt{\frac{4 \kappa}{d+2}} \, \nabla \cdot  \dot W_t \,,
\end{cases}
\end{equation}
where $W_t$ is a space/time white noise taking values in $\bbR^d$ and  $\bbW_t$ is a $d \times d$ matrix with coefficients given by independent white noises. We recall that~$P$ stands for  the Leray projection on divergence free vector fields.
Note that the noise is tuned so that the field has a covariance compatible with the invariance of   \eqref{eq: invariant measure}.
The equations \eqref{fluct-Stokes} should be understood in a weak sense, namely restricting to
  any pair of test functions~$ (\varphi,\psi)\in C^\infty(\T^d; \R^{d} \times \R)$
with~$ \nabla_x \cdot \varphi  = 0$
$$\begin{cases}
{\mathcal U}_\tau (\varphi)  = {\mathcal U}_0 ( e^{\nu \tau \Delta_x} \varphi)  
+ \sqrt{2 \nu} \displaystyle \int_0^\tau d\sigma  \dot \bbW_\sigma \big(   \nabla e^{\nu (\tau-\sigma) \Delta_x}  \,   \varphi  \big)  \\
\Theta_\tau (\psi) = \Theta_0 \big( e^{\kappa \tau\Delta_x} \psi \big) 
+ \sqrt{\frac{4 \kappa}{d+2}} \displaystyle\int_0^\tau d\sigma  \dot W_\sigma \big(   e^{\kappa (\tau-\sigma) \Delta_x} \nabla \psi  \big) \, .
\end{cases}
$$
%where $Q_s$ is the semigroup  associated with the anti-diffusion $\partial_s F = - \Delta_x F$.

We stress that  the fluctuations in \eqref{fluct-Stokes} exactly compensate the dissipation according to the fluctuation-dissipation principle.
In particular, both Gaussian processes are characterized by their covariances for $\sigma \leq t$
\begin{equation}
\label{fluct-Stokes covariance}
\begin{cases}
\bbE \Big( {\mathcal U}_\sigma (\varphi_1) {\mathcal U}_\tau (\varphi_2) \Big)  = \displaystyle
\int_{\T^d} dx  \,    \varphi_1 (x) \cdot e^{\nu (\tau-\sigma) \Delta_x}  \varphi_2 (x)\\ 
\bbE \Big( \Theta_\sigma (\psi_1) \Theta_\tau(\psi_2) \Big)  = \displaystyle
\frac{2}{d+2}
\int_{\T^d} dx \,  \psi_1 (x)  \; e^{\kappa (\tau-\sigma) \Delta_x} \psi_2 (x)\, .
\end{cases}
\end{equation}

%\subsection{The convergence result}
\bigskip

The main result of this paper is that both limits $\mu_\eps \to \infty$ with $\mu_\eps \eps^{d-1} = \alpha^{-1}$, and $\alpha \to 0$ can   be combined in order to derive fluctuating hydrodynamics directly from the dynamics of particles, thus solving Hilbert's sixth problem  in the particular case of   fluctuations of perfect gases at  equilibrium. 
\begin{Thm}
\label{main-thm}
Consider a system of hard spheres at equilibrium  in a~$d$-dimensional periodic box with~$d\geq 3$, with inverse mean free time $\alpha^{-1}  := \mu_\eps  \eps^{d-1} \leq ( \log \log \log \mu_\eps)$.
 Then, in the diffusive limit~$\mu_\eps \to \infty,  \alpha \to 0$,  the rescaled joint process 
$(   \xi^\eps_\tau )_{\tau \in [0,T]}$ defined in~{\rm(\ref{fluct-u-theta})} converges for any~$T>0$ in law to the solution of the fluctuating Stokes-Fourier equations {\rm(\ref{fluct-Stokes})}.
\end{Thm}

Recall that the microscopic dynamics is completely deterministic, so that stochasticity comes just as a consequence of the sensitivity of the particle system to the microscopic details of the initial configuration. 
In the low density regime, since the modulus of continuity  of  trajectories with respect to the initial configurations depends strongly on $\eps$, there is a strong  instability as~$\eps \to 0$, which generates some  ``spontaneous stochasticity"  encoded by the white noise in \eqref{eq: OU}.   The instability of the microscopic dynamics thus plays a key role in the structure of the noise in Theorem \ref{main-thm}. 
At variance, for one dimensional integrable systems, one expects that the dominant contribution is the transport of the initial fluctuations with some additional random shift in the large scale limit, this was pointed out recently in \cite{Ferrari_Olla} for the the hard rod system (see also \cite{Boldrighini_Wick}).  The white noise  in \eqref{eq: OU} preserves locally the hydrodynamic modes, however at diffusive time scales, 
it ultimately induces the local noise on the hydrodynamic projections
 \eqref{fluct-Stokes}.
Note that a   spontaneous generation of noise also holds for the diffusive limits of a tagged particle 
to a Brownian motion 
in an equilibrium hard sphere gas  \cite{BGSR1, BGSR3} (see also \cite{Er12} in the quantum case).
%\sout{
%We stress that the instability of the microscopic dynamics plays a key role in the structure of the noise in Theorem \ref{main-thm}. 
%{\color{red} 
%At variance, for one dimensional integrable systems, one expects that the dominant contribution is the transport of the initial fluctuations with some additional random shift in the large scale limit, this was pointed out recently in \cite{Ferrari_Olla} for the the hard rod system (see also \cite{Boldrighini_Wick}).}
%}

%%%%%%%%%%%%%%%%%%%%%%%%%%%%%%%%%%%%%%%%%%%%%%%%%%%%%%%%%%%%%%%%%%%%%%%%%%%%%%%%%%%%%%%%%%%%%%%

\section{The fluctuation field in the low density limit: state of the art, and strategy of proof}

The present paper relies on the  ``weak convergence" approach devised in \cite{BGSS3, BGSS4}
in order to prove the convergence of the fluctuation field to the solution of the fluctuating Boltzmann equation~\eqref{eq: OU}.
The proofs of~\cite{BGSS3,BGSS4} are quantitative, and 
 the  important parameter is the number of collisions, which is proportional to the observation time and inversely proportional to the mean free time $\alpha$. 
Thus, the (diffusive) observation time~$T/\alpha $ and the parameter~$\alpha^{-1}$ can  be chosen slowly diverging  with~$\mu_\eps$, for instance as~$O(\log\log \log \mu_\eps)$.  
This will allow us to reach the diffusive regime described in Section \ref{sec: Fluctuating hydrodynamics}.
In the rest of this section, we gather the results of~\cite{BGSS3, BGSS4} we shall be using here. We refer to those papers for proofs --- see also~\cite{BGSS5} for an overview.

For the sake of clarity, we will use the following notations for 
the different time scales described in the previous section~: 
\begin{equation}
\label{eq: time scale}
\begin{aligned}
&\text{kinetic scale : } t =  \alpha t_{kin} \hbox{ with } t_{kin} = O(1) , \quad 
\text{acoustic scale : }   t = O(1) , \\
& 
\text{diffusive scale : }  t = \tau  / \alpha \hbox{ with } \tau = O(1). 
\end{aligned}
\end{equation}

\subsection{Convergence of the covariance for diffusive times}

In the analysis of the fluctuation field for diffusive times, the  first step  is to study the asymptotic behaviour  of the time-rescaled covariance
\begin{equation}
\label{cov-def}
\Cov_\eps \left(\frac\tau\alpha ,g^0,h\right) := \bbE_\eps \Big(  \zeta^\eps_{0} (g^0)  \, \zeta^\eps_{\tau} (h)  \Big)
\end{equation}
 as $\mu_\eps \to \infty$, $\mu_\eps \eps^{d-1} = \alpha^{-1}$. 
 The following result states that  this covariance   is well approximated on $\R^+$ by $\displaystyle \int \mathcal M g_\alpha (\frac \tau \alpha ) h dxdv$ where $g_\alpha $ is the solution of the scaled  linearized Boltzmann equation~(\ref{eq:scaled- lbeq}) starting from $g^0\in L^2_\cM$, defined by the norm
\begin{equation}
\label{eq: L2 M}
 \|g\|_{L^2_\cM}:= \Big( \int_{ \T^d\times \R^d} |g|^2 \, \cM dxdv\Big)^\frac12
  \, .
\end{equation}
%We also recall (see~\cite{S2,BGSS3} for instance) that the space~$L^2_\cM$ is well suited to the fluctuating setting insofar as
%\begin{equation}
%\label{eq: moment ordre 2}
% \forall h \in L^2_\cM\, , \quad \bbE_\eps \big (\zeta^\eps_t (h)^2 \big)^\frac12 \leq C \|h\|_{L^2_\cM}\,,\quad C >0
%\, .
%\end{equation}
%
 
   \begin{Thm}[\cite{BGSS3}, {\bf Linearized Boltzmann equation}]
   \label{thmCov}
Consider a system of hard spheres at equilibrium  in a~$d$-dimensional periodic box with~$d\geq 3$.  Let~$g^0$ and~$h$ be two Lipschitz functions on $\T^d \times \R^d$ and let~$g_\alpha$ be the  unique solution in~$L^\infty(\R^+; L^2_{\mathcal M})$ to~{\rm(\ref{eq:scaled- lbeq})} associated with the initial data~$g^0$.
Then, in the low density regime~$\mu_\eps \to \infty$, $\mu_\eps \eps^{d-1} = \alpha^{-1} \leq  \log \log \log \mu_\eps $, the covariance of the fluctuation field~$\left(\zeta^\eps_{\tau/\alpha} \right)_{\tau\geq 0}$ defined by~{\rm(\ref{cov-def})} satisfies the following estimate: for any~$ T>0$  such that $(T/\alpha^2 )  \ll (\log   \log \mu_\eps )^\frac14$,
\begin{equation}
\label{convergence-rate}
\begin{aligned}
\sup_{\tau\in [0,T ]} & \left| \Cov_\eps \left(\frac \tau\alpha ,g^0,h\right) - \int  g_\alpha  (\frac \tau\alpha ) h\, \cM dxdv\right|\\
 & \leq C\| h\|_{W^{1,\infty}} \| g^0\|_{W^{1, \infty}}
 \left(\frac{C T   } {\alpha^2} \right)^{3/2} 
(\log   \log \mu_\eps )^{-1/4}  \,.
 \end{aligned}
\end{equation}
\end{Thm}
\begin{Rmk} 
In accordance with the diffusive scaling,
this estimate depends on $T/\alpha^2$, which is the ratio between the observation time $T/\alpha$ and the mean free time $\alpha$.
\end{Rmk}

\subsection{Convergence of higher order moments for diffusive times}

The next step is to prove that the   process $ \zeta^\eps_{\tau/\alpha}$  is asymptotically Gaussian when $\mu_\eps \to \infty$ and~$\mu_\eps \eps^{d-1} =\alpha^{-1} \to \infty$. This
boils down to showing that the moments are determined by the covariances according to Wick's rule
\begin{equation} 
\label{eq: appariement} 
\lim_{\mu_\eps \to \infty\atop \alpha  \to 0 } \Big| \bbE_\eps \Big[  \zeta^\eps_{\tau_1/\alpha} (h_1) \;  \dots 
\;   \zeta^\eps_{\tau_p/\alpha} (h_p) \Big] 
-  \sum_{\eta \in {\mathfrak S}_p^{\text{pairs}}}
\ \prod_{\{i,j\}  \in \eta} \bbE_\eps \Big[ \zeta^\eps_{\tau_i/\alpha} (h_i)  \; 
 \zeta^\eps_{\tau_j/\alpha} (h_j)  \Big] \Big| = 0\,  ,
\end{equation}
uniformly in $\tau_1, \dots , \tau_p \in [0, T]$, 
where ${\mathfrak S}_p^{\text{pairs}}$ is the set of partitions of $\{1, \dots , p \}$ made only of pairs. 
Notice that if $p$ is odd then ${\mathfrak S}_p^{\text{pairs}}$ is empty and the product of the moments is asymptotically 0.

\begin{Thm} [\cite{BGSS4}, {\bf Gaussian process}]
\label{thmTCL}
Consider a system of hard spheres at equilibrium  in a~$d$-dimensional periodic box with~$d\geq 3$.  Let~$(h_i)_{1 \leq i \leq p}$ be a family of~$p$ bounded functions on~$\T^d \times \R^d $.
Then, in the low density regime~$\mu_\eps \to \infty$, $\mu_\eps \eps^{d-1} = \alpha^{-1} \leq  \log \log \log \mu_\eps $, the  fluctuation field~$\left(\zeta^\eps_{\tau/\alpha} \right)_{\tau \geq 0}$ defined by~{\rm(\ref{eq: fluctuation field})} is almost Gaussian  in the sense that  for any~$T>0 $ satisfying~$(T/\alpha^2) ^{\frac{2p-1}2}      \ll (\log   \log \mu_\eps )^\frac14$, there holds uniformly in~$\tau_1, \dots \tau_p \in [0, T]$, 
\begin{equation}
\label{convergence-Wick}
\begin{aligned}
 &\left| \bbE_\eps \Big[ \zeta^\eps_{\tau_1/\alpha} (h_1)  \;  \dots 
\; \zeta^\eps_{\tau_p/\alpha} (h_p) \Big] 
-  \sum_{\eta \in {\mathfrak S}_p^{\text{pairs}}}
\ \prod_{\{i,j\}  \in \eta} \bbE_\eps \left[  \zeta^\eps_{\tau_i/\alpha} (h_i)  \; 
 \zeta^\eps_{\tau_j/\alpha} (h_j)  \right] \right| \\
&\qquad \qquad \qquad \qquad 
\leq  \Big( \prod_{i = 1} ^p \| h_i \|_{L^\infty} \Big)   \Big(\frac {CT} {\alpha^2}\Big) ^{(2p-1)/2} (\log   \log \mu_\eps )^{-1/4}  \, . \end{aligned}
\end{equation}
\end{Thm}

\subsection{Tightness in the kinetic regime}

Finally for processes which depend on a continuous variable (the time  variable in our setting), the convergence of time marginals is not enough to characterize the convergence in law:  possible oscillations with respect to time need to be under control (see \cite[Theorem 13.2 page 139]{Billingsley}). For the fluctuation field~$\zeta^\eps$, this tightness property has been obtained  for short kinetic times, but actually since the equilibrium measure is invariant under the dynamics, a union bound  provides the tightness on any finite kinetic time, i.e.  times of order~$O(\alpha)$.

For times much longer than  kinetic times, we actually do not expect the process $\zeta^\eps_t $ to be tight. Since the covariance $\Cov_\eps \left(t ,g^0,h\right)$ is close to the solution  of the scaled  linearized Boltzmann equation {\rm(\ref{eq:scaled- lbeq})}, we expect to see a fast relaxation process with rate $O(\frac1\alpha)$, meaning that only the hydrodynamic part of $g_\alpha$ can be compact for $t= O(1)$.
Going to diffusive times $ t = \tau /\alpha$, we also expect to have acoustic waves producing fast oscillations, meaning that only the non oscillating  hydrodynamic  part of $g_\alpha$ can be compact for $\tau = O(1)$.
Nevertheless, after projecting on the non oscillating modes,
we are going to show in Section \ref{tightness} that the process $( \xi^\eps_\tau)_{\tau \geq 0}$ defined by  \eqref{fluct-u-theta} is tight on the diffusive scale.

\subsection{Strategy of the proof of Theorem~\ref{main-thm}}

In view of deriving   fluctuating hydrodynamic equations and proving Theorem~\ref{main-thm}, the strategy is now straightforward: we   consider the rescaled  fluctuation field $  \xi^\eps_{\tau }$ projected on hydrodynamic, non oscillating modes (recall~(\ref{fluct-u-theta})), and  check that with such test functions and  this scaling in time, Gaussianity (Theorem~\ref{thmTCL})  and tightness   still hold, and that the   covariance asymptotically converges to the solution to the Stokes-Fourier equation. 
Note that the projection~(\ref{fluct-u-theta}) leads to considering test functions which are  unbounded in $v$ and therefore   there are some   technical issues when applying Theorems~\ref{thmCov} and~\ref{thmTCL}.   These 	are dealt with in Section~\ref{cov}, thanks to a cut-off in energies introduced in Section~\ref{truncation}.  The tightness of the process on the diffusive time-scale  is derived in Section~\ref{tightness}.

% \begin{Rmk}
% The papers~\cite{BGSS3, BGSS4} are  restricted to dimensions~$d \geq 3$, whence the restriction in Theorem~{\rm\ref{main-thm}}: the case~$d=2$ is expected to be true but the proofs of~\cite{BGSS3, BGSS4} would require some technical adjustements.
% \end{Rmk}

%%%%%%%%%%%%%%%%%%%%%%%%%%%%%%%%%%%%%%%%%%%%%%%%%%%%%%%%%%%%%%%%%%%%%%%%%%%%%%%%%%%%%%%%%%%%%%%

\section{Finite time marginals}
%\section{Hydrodynamic limits}
\label{Hydro limits}

In this section, we are going to characterize the limiting law of the process by proving the following result.  We set from now on$$  \alpha^{-1}   =  \log \log \log \mu_\eps\, .$$
\begin{Prop}
\label{Prop: marginales temps}
For arbitrary times $\tau_1, \dots, \tau_L$ and test functions $\phi^{(1)} = (\varphi^{(1)}, \psi^{(1)}),\dots$ and~$\phi^{(L)} = (\varphi^{(L)}, \psi^{(L)})$ chosen as in \eqref{def test function}, the time marginals $\big( \xi^{\eps}_{\tau_\ell} \big( \phi^{(\ell)} \big)  \big)_{\ell \leq L}$ converge in law to the limiting process
$\big(  {\mathcal U}_{\tau_\ell }  (\varphi^{(\ell)}),   { \Theta}_{\tau_\ell} (\psi^{(\ell)})  \big)_{\ell \leq L}$ as $\mu_\eps$ tends to infinity.
\end{Prop}

\subsection{Truncated hydrodynamic fields}
\label{truncation}

To prove that the limit is Gaussian, Theorem \ref{thmTCL} cannot be used directly with the process 
$(\xi^{\eps}_\tau)_{t \geq 0}$ as the test functions are unbounded in $L^\infty$ due to divergences in the velocities. Thus an intermediate cut-off process needs to be introduced.
Let us fix an energy cut-off~$R\gg 1$ to be determined (see~(\ref{scaling}) below).  Recalling~(\ref{fluct-u-theta}), we define the modified joint process~$  \bar\xi^{\eps}_\tau$ as follows. For any test function~$\phi$  as in~(\ref{def test function}), we set
\begin{equation}
\label{u-theta-defR}
\begin{aligned}
   \bar\xi^{\eps}_{\tau} (\phi)
:=  \zeta^\eps_{\tau/\alpha} \Big( \chi \Big({|v|^2 \over  R}\Big)
 \varphi \cdot v \Big)+   \zeta^\eps_{\tau/\alpha} \Big(  \chi \Big({|v|^2 \over  R}\Big) \psi \big(  {|v|^2 \over d+2}-1\big)\Big) \, ,
 \end{aligned}
\end{equation}
where $\chi$ is a smooth cut-off function with compact support
$$ \chi_{|[0,1]} \equiv 1, \qquad \chi_{|[2,+\infty[} \equiv 0\,.$$
We choose    $R$ depending on $\e$ and converging to $\infty$ as  $\mu_\eps \to \infty$ as follows
\begin{equation}
\label{scaling} 
R =  \alpha^{-1}   =  \log \log \log \mu_\eps\,.
\end{equation}
Note that the  test functions 
$$ \bar h:=  \Big( \varphi \cdot v + \psi   \big(  {|v|^2 \over d+2}-1\big)\Big) \chi \left({|v|^2 \over  R}\right)$$
  are smooth and bounded thanks to the cut-off in $v$:
\begin{equation}
\label{h-bound}
 \| \bar h\|_{W^{1,\infty}_{x,v} }  \leq CR^2 (   \| \varphi\|_{W^{1,\infty}_x} +\| \psi\|_{W^{1,\infty}_x} ) \,.
 \end{equation}
%we infer in particular
The process  $\bar\xi^{\eps}_\tau$ is a good approximation of $\xi^{\eps}_\tau$ when $R\to \infty$. 
\begin{Lem}
\label{Lem: cut-off approximation}
Setting $ \xi^{\eps,>}_{\tau} :=  \xi^{\eps}_\tau -   \bar\xi^{\eps}_{\tau}$ then
for all~$1 \leq q < \infty$ and for~$\eps$ small enough
\begin{equation}
\label{cutoffxi}  
\bbE_\eps  \left[  \left(  \xi^{\eps,>}_{\tau}  (\phi) \right)^q \right] \leq C_q \| \phi \|^q_{L^q (\T^d )} e^{-R/4} 
\,,
\end{equation}
 with $L^q_{\cM}$ defined as in  \eqref{eq: L2 M}
Furthermore, one has also 
\begin{equation}
\label{eq: borne Lq cutoffxi}  
\bbE_\eps  \left[  \left( \xi^{\eps}_\tau  (\phi) \right)^q \right] \leq C_q \| \phi \|^q_{L^q (\T^d )}
\quad \text{and} \quad 
\bbE_\eps  \left[  \left( \bar\xi^{\eps}_{\tau}  (\phi) \right)^q \right] \leq C_q \| \phi \|^q_{L^q (\T^d )}.
\end{equation}
\end{Lem}
As a consequence, the convergence in law of $\big( \bar \xi^{\eps}_{\tau_\ell} \big( \phi^{(\ell)} \big)  \big)_{\ell \leq L}$
(derived in Proposition \ref{Prop: marginales temps cut-off} below) will imply the convergence in law of $\big( \xi^{\eps}_{\tau_\ell} \big( \phi^{(\ell)} \big)  \big)_{\ell \leq L}$, i.e. Proposition \ref{Prop: marginales temps}.

\begin{proof}[Proof of Lemma {\rm\ref{Lem: cut-off approximation}}]

Recall (see Proposition A.1 in~\cite{BGSS3}) that for any $\eps$ small enough, 
the following holds under the equilibrium measure for any  function $h$ 
\begin{equation}
\label{eq: moment ordre pR}
 \bbE_\eps \Big( \big( \xi^{\eps}_\tau (h)\big) ^q \Big) \leq C_q\| h\|_{L^q_{\cM}}^q  \,,
\end{equation}
with $1 \leq q < \infty$.
Since for~$R \geq  1$
\begin{equation}
\label{cutoff}
\begin{aligned}
& \Big\|   \varphi \cdot v \Big( \chi \big( {|v|^2 \over  R}\big) - 1\Big) \Big\| _{L^q_\cM} ^q
\leq C\| \varphi \|_{L^q (\T^d )} ^q  e^{-R/4} \\
 &  \Big\|  \psi \big(  {|v|^2 \over d+2}-1\big)   \Big( \chi \big({|v|^2 \over  R}\big)  - 1\Big) \Big\| _{L^q_\cM} ^q
 \leq C\| \psi\|_{L^q (\T^d)}^q  e^{-R/4} \, , 
 \end{aligned}
\end{equation}
we find~(\ref{cutoffxi}). 
For the same reason \eqref{eq: borne Lq cutoffxi}  holds. 
This completes  Lemma \ref{Lem: cut-off approximation}.
%We infer in particular that  the moments of the energy-truncated fluctuation field are bounded by : 
%Moreover thanks to~(\ref{eq: moment ordre 2}) there also holds
%\begin{equation}
%\label{eq: moment ordre 2R}
% \bbE_\eps \Big( \big(   \bar\xi^{\eps}_{\tau} (\phi)\big) ^2 \Big) \leq C \|\bar h\|_{L^2_\cM}^2 \leq C'\|\phi\|_{L^2 (\T^d )}^2\,.
%\end{equation}
\end{proof}

\subsection{Covariance of the hydrodynamic fields}
\label{cov}

\begin{Prop}
\label{Prop: marginales temps cut-off}
For arbitrary times $\tau_1, \dots, \tau_L$ and test functions $\phi^{(1)} = (\varphi^{(1)}, \psi^{(1)}),\dots$ and~$\phi^{(L)} = (\varphi^{(L)}, \psi^{(L)})$ chosen as in \eqref{def test function}, the time marginals $\big( \bar \xi^{\eps}_{\tau_\ell} \big( \phi^{(\ell)} \big)  \big)_{\ell \leq L}$ converge in law to  the limiting process
$\big(  {\mathcal U}_{\tau_\ell }  (\varphi^{(\ell)}),   { \Theta}_{\tau_\ell} (\psi^{(\ell)})  \big)_{\ell \leq L}$  as $\mu_\eps$ tends to infinity.
\end{Prop}
Combined with the approximation Lemma \ref{Lem: cut-off approximation}, this completes the proof of Proposition \ref{Prop: marginales temps}.
The proof of Proposition \ref{Prop: marginales temps cut-off} is split into two parts, first a control of the limiting covariance and then the derivation of Wick's rule to prove that the limiting process is Gaussian.

\medskip

\noindent
{\bf Step 1. Control of the covariance}.
Let us define the hydrodynamic, non oscillating projections
\begin{equation}
\label{eq: well prepared}
\begin{aligned}
& g^0 (x,v) := \left( u_0(x) \cdot v+ \theta_0 (x) {|v|^2-(d+2) \over 2}  \right)\, ,\\
 &  h(x,v) :=\left(  \varphi (x) \cdot v + \psi(x) \left(  { |v|^2 \over d+2 }-1 \right)  \right)\, ,
 \end{aligned}
 \end{equation}
 for some smooth divergence free vector fields $u_0, \varphi$, and some smooth functions $\theta_0, \psi$. 
The scaling in $g^0,h$ has been tuned asymmetrically so that the initial covariance is given by
 \begin{equation*}
 \bbE_\eps \left[   \bar  \xi^{\eps}_0 ( \phi_0 )    \bar \xi^{\eps}_0 ( \phi ) \right]  \longrightarrow \int  ( u(\tau) \cdot \varphi + \theta(\tau) \psi ) dx\,,  \quad \mu_\eps \to \infty\, .
\end{equation*}
We are   going to study the covariance of the joint process $\bar \xi^\eps_\tau $ by applying Theorem \ref{thmCov} with 
\begin{equation}
\label{eq: truncated initial data}
\bar g^0 (x,v) := g^0(x,v) \chi \left({|v|^2 \over  R}\right) \, , \quad 
\bar  h(x,v) =  h(x,v) \chi \left({|v|^2 \over  R}\right)\, .
 \end{equation}
   Setting
   $$\displaystyle \phi_0:=(u_0,\frac{d+2}2\theta_0)\, , \quad 
\phi :=(\varphi,\psi)\, ,
$$
we plug the bounds~(\ref{h-bound}) on the test functions into the estimate~(\ref{convergence-rate}) of Theorem \ref{thmCov}, and recalling the definition~(\ref{u-theta-defR}) of the truncated rescaled fluctuation field, we
 obtain that  for any $  T>0$ such that $(T/\alpha^2)   \ll (\log \log \mu_\eps)^{-1/2}$,
\begin{equation}
\label{approximate-covarianceR}
\begin{aligned}
\sup_{t \in [0,T ]} & \left| \bbE_\eps \left[    \bar\xi^{\eps}_0 ( \phi_0 )     \bar\xi^{\eps}_\tau ( \phi ) \right]  - \int \cM \tilde g_\alpha (t) \bar h dxdv\right|\\
 & \leq CR\| \phi_0 \|_{W^{1,\infty}} \| \phi\|_{W^{1, \infty}}   \left(\frac{C T  } {\alpha^2} \right)^{3/2}  (\log \log \mu_\eps)^{-1/4} \,,
 \end{aligned}
\end{equation}
where $\tilde g_\alpha $ is the solution to  the time-rescaled equation
\begin{equation}
\label{diffusive-beq}
\alpha  \partial_\tau \tilde g_\alpha +v\cdot \nabla_x \tilde g_\alpha  +\frac1\alpha \cL \tilde g_\alpha = 0\, , \qquad \tilde g_{\alpha |t = 0} = \bar g^0  \;.
\end{equation} 
%Thanks to the bounds~(\ref{eq: moment ordre 2}),~(\ref{eq: moment ordre 2R}) on the moments of order two and  to the convergence of the remainder given in~(\ref{cutoffxi}),  a Cauchy-Schwarz inequality implies 
%$$
% \left| \bbE_\eps \left[    \xi^{\eps}_0 ( \phi_0 )     \xi^{\eps}_\tau ( \phi ) \right]  -\bbE_\eps \left[    \bar\xi^{\eps}_0 ( \phi_0 )     \bar\xi^{\eps}_\tau ( \phi ) \right] \right| \longrightarrow 0 \, , \quad \mu_\eps \to \infty
%$$
%so with~(\ref{approximate-covarianceR})
%\begin{equation} 
%\label{approximate-covariance}
%\sup_{t \in [0,T ]} \left| \bbE_\eps \left[    \xi^{\eps}_0 ( \phi_0 )     \xi^{\eps}_\tau ( \phi ) \right]  - \int \cM \tilde g_\alpha (t) h dxdv\right|  \longrightarrow 0 \, , \quad \mu_\eps \to \infty \, .
%\end{equation}
To conclude to the convergence of the covariance as~$\alpha \to 0$, we just need to identify the limit of~$\displaystyle \int \cM \tilde g_\alpha (\tau) \bar h dxdv$.

\medskip
The starting point for the study of hydrodynamic limits of the linearized Boltzmann equation (\ref{diffusive-beq})  is the scaled energy inequality
\begin{equation}\label{scaled energy inequality}
\frac12  \left\| \tilde g_\alpha(\tau) \right\|^2_{L^2(\cM dvdx)} +{1\over \alpha^2 }  \int_0^\tau  \int \tilde g_\alpha \cL \tilde g_\alpha (\tau' ) \cM dvdx d\tau'\leq  \frac12 \| \bar g^0\|_{L^2(\cM dv dx)}^2\,.
 \end{equation}
Recall (see~\cite{grad, Hilbert}) that the linearized  collision
operator $\cL$ with hard sphere cross section defined by (\ref{eq:LBCO})  is a nonnegative unbounded 
self-adjoint operator on $L^2(\cM dv)$  with domain
$$
\cD(\cL)=L^2\big(\R^d;(1+|v|)  \cM dv\big)
$$
and nullspace
$$
\Ker(\cL)=\Span\Big\{1,v_1,\dots,v_d,|v|^2\Big\}\,.
$${ In particular we recover  from~(\ref{scaled energy inequality})  the uniform  $L^2$  bound
$$
 \left\| \tilde g_\alpha(\tau) \right\|_{L^2(\cM dvdx)}  \leq \| \bar g^0\|_{L^2(\cM dvdx)}  \leq \|  g^0\|_{L^2(\cM dvdx)} \,.
$$
This bound   implies that  there is $g \in L^\infty_\tau( L^2(\cM dvdx))$ such that, up to extraction of a subsequence,
\begin{equation}
\label{L2-cv}
\tilde g_\alpha  \rightharpoonup g \hbox{ weakly in } L^2_{\rm loc} (d\tau,L^2(\cM dvdx))\,.
\end{equation}
Moreover the following coercivity estimate holds : there exists $C>0$ such that, for 
each~$g$ in~$\cD(\cL)\cap(\Ker(\cL))^\perp$
\begin{equation}
\label{coercivity}
\int g\cL g(v)\cM (v)dv\geq C\|g\|_{L^2((1+|v|)  \cM  dv)}^2\,.
\end{equation}
The dissipation  thus  further provides  
$$ 
\left\| \tilde g_\alpha -\Pi \tilde g_\alpha\right\|_{L^2(  (1+|v|) \cM dvdxdt)} = O(\alpha) \, ,
$$
where~$\Pi$ denotes the orthogonal projection onto~$\Ker(\cL)$ in~$L^2(\cM dvdx)$.
We deduce from the previous estimate that
\begin{equation}
\label{ansatz}
 g (\tau,x,v) =\Pi g(\tau,x,v) \equiv \rho(\tau,x) + u(\tau,x)\cdot v +\theta(\tau,x) {|v|^2- d \over 2}\, \cdotp
 \end{equation}
 It remains to compute the equations on~$\rho$,~$u$ and~$\theta$.  Denoting   $\displaystyle \la g \ra := \int g \cM dv $
 and recalling~(\ref{diffusive-beq}), 
  the moment equations state
$$
\begin{aligned}
& \alpha \d_\tau \la \tilde g_\alpha \ra +  \nabla_x \cdot \la \tilde g_\alpha  v\ra=0 \, ,\\
& \alpha \d_\tau \la \tilde g_\alpha v\ra + \nabla_x \cdot \la \tilde g_\alpha  v\otimes v\ra=0 \, ,\\
& \alpha \d_\tau \la \tilde g_\alpha |v|^2 \ra + \nabla_x \cdot \la \tilde g_\alpha  v|v|^2\ra=0\,.
\end{aligned}
$$
Using (\ref{L2-cv}) and (\ref{ansatz}) we deduce from the first two equations that
\begin{equation}
\label{eq: non oscillating}
\nabla_x \cdot u = 0 \, , \quad\nabla_x (\rho+\theta) =0\,,
\end{equation}
referred to as the incompressibility and Boussinesq constraints.   We thus have
\begin{equation}
\label{new ansatz}
 g (\tau,x,v)   =   u(\tau,x)\cdot v +\theta(\tau,x) {|v|^2- (d+2) \over 2}\,, \quad \nabla_x \cdot u = 0\, . \end{equation}
Note that, up to the cut-off in $v$ which can be removed with a small error thanks to~(\ref{cutoff}),  the test function $\bar h$ is in the kernel of acoustic operator. It follows that  we only need to characterize the mean motion, namely derive the equations for~$ P \la \tilde g_\alpha v\ra$ and~$\la \tilde g_\alpha (|v|^2 -d-2)\ra$:
$$
\begin{aligned}
& \d_\tau P \la \tilde g_\alpha v\ra +{1\over \alpha} P\nabla_x \cdot \la \tilde g_\alpha ( v\otimes v-\frac1d |v|^2\mbox{Id})\ra=0 \, ,\\
& \d_\tau  \frac1{d+2} \la \tilde g_\alpha (|v|^2 -d-2)\ra +{1\over \alpha}\nabla_x \cdot \la \tilde g_\alpha  \frac1{d+2}v(|v|^2-d-2)\ra=0\,,
\end{aligned}
$$
where we recall that $P$ is the Leray projection on divergence free vector fields.
Define the kinetic momentum flux $\displaystyle A(v) := v\otimes v-\frac{1}{d} |v|^2 \mbox{Id}$ and the kinetic energy flux~$\displaystyle B(v) :=  \frac1{2}  v(|v|^2-d-2)$. As~$A, B $ belong to~$ (\Ker \cL)^\perp$, and $\cL$ is a Fredholm operator, there exist
pseudo-inverses $\tilde A, \tilde B$ in~$ (\Ker \cL)^\perp$ such that $A = \cL \tilde A$ and $B = \cL \tilde B$. Then,
$$
\begin{aligned}
& \d_\tau  P \la \tilde g_\alpha v\ra + {1\over \alpha}  P\nabla_x \cdot \la (\cL \tilde g_\alpha) \tilde A \ra=0 \, ,\\
& \frac1{d+2} \d_\tau  \la \tilde g_\alpha (|v|^2 -d-2)\ra +{1\over \alpha} \frac2{d+2}\nabla_x \cdot \la (\cL \tilde g_\alpha ) \tilde B\ra=0\,.
\end{aligned}
$$
Using the equation
\begin{equation}
\label{eqL}
{1\over \alpha} \cL \tilde g_\alpha = -v\cdot \nabla_x \tilde g_\alpha -\alpha \d_\tau \tilde g_\alpha 
\end{equation}
we get
\begin{equation}
\label{time-control}
\begin{aligned}
& \d_\tau  P \la \tilde g_\alpha v\ra -   P\nabla_x \cdot \la(v\cdot \nabla_x+\alpha \d_\tau )  \tilde g_\alpha \tilde A \ra=0 \, ,\\
& \frac1{d+2}  \d_\tau \la \tilde g_\alpha (|v|^2 -d-2)\ra -\frac2{d+2}\ \nabla_x \cdot \la(v\cdot \nabla_x+\alpha \d_\tau ) \tilde g_\alpha  \tilde B\ra=0\,.
\end{aligned}
\end{equation}
Then, plugging the Ansatz (\ref{ansatz}), and taking limits in the sense of distributions, we get the Stokes-Fourier equations
\begin{align*}
& \d_\tau  u - \nu  \Delta_x u=0 \, , \quad \nabla_x \cdot  u = 0\, ,\\
& \d_\tau \theta  - \kappa  \Delta_x \theta =0\,,
\end{align*}
with initial data as in \eqref{eq: initial data}
\begin{equation*}
u_{|\tau=0}(x) := P \int g_0(x,v) v \cM(v)  \, dv \, , \quad \theta_{|\tau=0}(x) :=  \int g_0 (x,v) \Big(  \frac{|v|^2}{ (d+2)}-1\Big)   \cM(v) \, dv ,
\end{equation*}
and 
where the diffusion coefficients are given by
\begin{equation}
\label{mu-kappa}
\nu := \frac{1}{(d-1)(d+2)} \la A : \tilde A \ra 
\quad \text{and} \quad 
\kappa := \frac{2}{d(d+2)} \la B \cdot \tilde B \ra\, .
\end{equation}
We therefore end up  with the following convergence as $\alpha \to 0$
\begin{equation}
\label{SF-weaklimit}
\int \cM  \tilde g_\alpha (\tau) \bar h dxdv \longrightarrow \int  ( u(\tau) \cdot \varphi + \theta( \tau ) \psi ) dx\,.
\end{equation}
Returning to \eqref{approximate-covarianceR},  we have proved that
\begin{equation}
\label{eq: cov tronquee limite}
\sup_{\tau \in [0,T ]} \bbE_\eps \left[   \bar  \xi^{\eps}_0 ( \phi_0 )    \bar \xi^{\eps}_\tau ( \phi ) \right]  \longrightarrow \int  ( u(\tau) \cdot \varphi + \theta(\tau) \psi ) dx\,,  \quad \mu_\eps \to \infty\, .
\end{equation}

\begin{Rmk}
Since  the initial data $g^0$ is well-prepared, both the  purely kinetic component and the fast oscillating  acoustic waves are negligible, so   the convergence  of $\tilde g_\alpha$ can be shown actually to hold in  strong sense. Using energy methods, it is even possible to obtain a rate of convergence for~{\rm(\ref{SF-weaklimit})}. 
%First, one defines a sharp approximation (writing~$\Pi^\perp:= \mbox{Id} - \Pi$  and using~{\rm(\ref{eqL})})
%$$ \begin{aligned}
%g_{app} (\tau,x,v) := &u(\tau,x) \cdot v + \theta (\tau,x) {|v|^2 - (d+2) \over 2} -\alpha \big(\nabla_x u(\tau,x) : \tilde A (v)+ \nabla_x \theta(\tau,x) \cdot \tilde B (v)\big)\\
%&+\alpha^2 \left( D^2_x u(\tau,x) :\cL^{-1} \Pi_\perp(v\tilde A )+ D^2_x \theta(\tau,x) : \cL^{-1} \Pi_\perp( v\tilde B) \right)  \,.
%\end{aligned}$$
%A straightforward computation leads then to 
%$$ \d_\tau g_{app} +\frac1\alpha v\cdot \nabla_x g_{app} +\frac1{\alpha^2} \cL g_{app} = O(\alpha \| (u,\theta)\|_{W^{3,\infty}_{t,x}})\,$$
%which provides
%$$
%\begin{aligned}
%& \left\|   \tilde g_\alpha (\tau) - g_{app}   \right\|^2_{L^2(\cM dvdx)} +{1\over \alpha^2 }  \int_0^\tau  \int \left(  \tilde g_\alpha   - g_{app} \right) \cL \left( \tilde g_\alpha  -g_{app} \right) (\tau')  \cM dvdx d\tau'\\\
% &\quad  \leq e^\tau \left( \| g^0- g_{app|\tau = 0} \|_{L^2(\cM dv)}^2+C\alpha\| (u,\theta)\|_{W^{3,\infty}} \right) \,.
% \end{aligned}$$
\end{Rmk}

\bigskip

\noindent
{\bf Step 2. Wick's rule}

%Our goal is to prove that
%\begin{equation}
%\label{Wick}
%\begin{aligned}
% \Big| \bbE_\eps \Big[ \xi^{\eps}_{t_1} \big( \phi^{(1)} \big)\;  \dots 
%\;  \xi^{\eps}_{t_p} \big( \phi^{(p)} \big)\Big]   - \sum_{\eta \in {\mathfrak S}_p^{\text{pairs}}}
%\ \prod_{\{i,j\}  \in \eta} \bbE_\eps \left[  \xi^{\eps}_{t_i} \big( \phi^{(i)} \big) \xi^{\eps}_{t_j} \big( \phi^{(j)} \big)  \right] \Big|  \longrightarrow 0 \,, \quad \mu_\eps \to \infty\,,
%\end{aligned}
%\end{equation}
%for arbitrary times $t_1, \dots t_p$ and test functions $\phi^{(1)},\cdots,\phi^{(p)}$ chosen as in \eqref{def test function}.
Consider~$p$ times~$\tau_1,\dots,\tau_p$, possibly repeated. Thanks to the cut-off~(\ref{h-bound}), we can apply Theorem \ref{thmTCL} to obtain 
\begin{equation}
\label{WickR}
\begin{aligned}
 &\Big| \bbE_\eps \Big[ \bar\xi^{\eps}_{\tau_1} \big( \phi^{(1)} \big)\;  \dots 
\;  \bar\xi^{\eps}_{\tau_p} \big( \phi^{(p)} \big)\Big]   - \sum_{\eta \in {\mathfrak S}_p^{\text{pairs}}}
\ \prod_{\{i,j\}  \in \eta} \bbE_\eps \left[  \bar\xi^{\eps}_{\tau_i} \big( \phi^{(i)} \big) \bar\xi^{\eps}_{\tau_j} \big( \phi^{(j)} \big)  \right] \Big| \\
&\qquad\qquad  \leq C_p R^p  \prod_{i = 1} ^p \| \phi^{(i)} \|_{L^\infty}     \Big(\frac {CT} {\alpha^2}\Big) ^{(2p-1)/2} (\log   \log \mu_\eps )^{-1/4}  \, . 
   \end{aligned}
\end{equation}
With the scaling condition (\ref{scaling}), 
we get that the right-hand side converges to 0 as $\mu_\eps \to \infty$ which implies
the asymptotic pairing of the moments of $ \bar\xi^{\eps}_{\tau}$.  
Since the limiting covariance is characterized by \eqref{eq: cov tronquee limite}, this completes Proposition \ref{Prop: marginales temps cut-off}.  \qed

\section{Tightness of hydrodynamic fields on diffusive time scales}
\label{tightness}

%To recast the statement of Theorem \ref{thmtight} for the  process $(\xi^\eps_\tau)_{\tau \geq 0}$ in a given time interval~$[0,T]$, 
Let us first introduce for any~$k \in {\mathbb Z}$  the Sobolev space $\bbH^k$ in $\T^d$ with the norm
\begin{equation}
\label{eq: Sobolev -k}
\| F \|_{k}^2 := \sum_{j\in \Z^d}  {\left(1+|j|^2\right)^k} \; | \hat F_j |^2 ,
\end{equation}
where $(\hat F_j)$ stand for the Fourier coefficients of $F$.
%This will play an analogous role to the space $\cH^k$  in \eqref{eq: Fourier-Hermite modes} (for test functions depending only on the space variables) 
%and for simplicity,  the norm is denoted in the same way.

\begin{Prop} 
\label{thmtight diffusive}
There exists $k>0$ such that, in the diffusive limit
$$
\mu_\eps \to \infty, \alpha \to 0, \ 
\quad \text{with} \quad 
\mu_\eps \eps^{d-1} = \alpha^{-1} \leq   \log \log \log \mu_\eps, 
$$ 
the  fluctuation field~$\left(\xi^\eps_\tau \right)_{\tau \geq 0}$ defined 
by~\eqref{fluct-u-theta}
 is tight in the Skorokhod space $D\left([0,T] , \bbH^{-k} \right)$. More precisely, 
\begin{equation}
\label{eq: tightness}
\begin{aligned}
&  \lim_{\delta \to 0^+} \lim_{\mu_\eps \to \infty} 
\bbP_\varepsilon \Big[ \sup_{ | \sigma - \tau| \leq \delta \atop s,\tau  \in [0,T] }
\big\| \xi^\varepsilon_\tau   - \xi^\varepsilon_ \sigma \big \|_{-k} \geq \delta'
\Big] = 0 \,, \qquad \forall \delta'>0 \,,  \\
& 
\lim_{A \to \infty} \lim_{\mu_\eps \to \infty} \bbP_\varepsilon \Big[ \sup_{ \tau  \in [0,T]}
\big \| \xi^\varepsilon_\tau \big \|_{-k} \geq A \Big] = 0\, .
\end{aligned}
\end{equation}
\end{Prop}

The tightness property for kinetic times  
%stated in Theorem~\ref{thmtight}
% is  obtained essentially by combining the following averaged time continuity (which is expressed in terms of moments  of the fluctuation field)
%\begin{equation}
%\label{eq: norme 4 t -s}
%\bbE \left( \big( \zeta^\eps_{\alpha t}  (h)  - \zeta^\eps_{\alpha s}  (h) \big)^4  \right)  
%\leq 
%C  \| \nabla h\|_{W^{1,\infty}}^4  \; |t-s| \left( |t-s| + \frac{1}{\mu_\eps}  \right)\, , 
%\end{equation}
relies on the Garsia-Rodemich-Rumsey inequality on the modulus of continuity
of a function $\varphi_\tau : [0,T] \to \R$, which we recall (\cite{Varadhan}): for $b \geq 4$
\begin{equation}
\label{GRR} 
\sup_{0 \leq \sigma,\tau \leq T \atop |\tau - \sigma| \leq \delta} 
\big| \varphi_\tau - \varphi_\sigma   \big|
\leq C\left(  \int_0^T \int_0^T d \sigma d \tau \;
 \frac{| \varphi_\tau - \varphi_\sigma|^b}{|\tau - \sigma|^\gamma} \right)^{1/b} 
 \delta^{\frac{\gamma -2}{b} }\,,\qquad
 \gamma\in ]2,3[\;.
\end{equation}
Because of collisions in the Newtonian dynamics, the fluctuation field $\xi^\eps$ has jumps and 
this inequality does not apply directly. 
%In~\cite{BGSS2} we proved a modified version of the estimate, which is adapted to the process with jumps. This is based on introducing a cutoff on small time differences~$|t-s| \leq \mu_\eps^{-7/3}$; indeed the small time fluctuations can be dealt with separately by using averaged estimates.
We therefore start by stating a modified inequality, whose proof is a slight adaptation of \cite{Varadhan}
which can be found in \cite{BGSS2} (see Proposition 6.2.4).
\begin{Prop}
\label{prop: Modified Garsia, Rodemich, Rumsey inequality}
Let $F : [0,T] \to \bbR$ be a given function and define for $a>0$, $b \geq 4$
\begin{equation}
\label{eq: bound GRR alpha 0} 
 B_a (F) : = 
\int_0^T \int_0^T d \sigma d \tau \;
 \frac{| F_\tau - F_\sigma |^b}{| \tau - \sigma |^\gamma} {\bf 1}_{| \tau - \sigma | > a}\,,\qquad
 \gamma\in ]2,3[\;.
\end{equation}
%If $\gp$ satisfies
%\begin{equation}
%\label{eq: modulus short} 
%\sup_{0 \leq s,t\leq {T^\star} \atop |t - s| \leq  2 \alpha} \big| \gp_t - \gp_s \big|
%\leq \nu\, ,
%\end{equation}
Then the  modulus of continuity of $F$ is controlled by
\begin{equation}
\label{eq: modulus GRR cutoff} 
\sup_{0 \leq \sigma , \tau \leq T \atop | \tau - \sigma | \leq \delta} \big| F_\tau - F_\sigma  \big|
\leq 2 \sup_{0 \leq \sigma, \tau \leq T \atop |\tau - \sigma| \leq  2 a} \big| F_\tau - F_\sigma \big| 
\; + \; C  B_a (F)^{ \frac{1}{b} } \;  \delta^{\frac{\gamma-2}{b} }\,  .
\end{equation}
\end{Prop}

\medskip

\begin{proof}[Proof of Proposition \ref{thmtight diffusive}]
To prove  the tightness of the joint process $( \xi^{\eps}_\tau )_{\tau \geq 0}$
in $D( [0,T], \bbH^{-k})$ for some  $k$  large enough, we shall tune the parameter $a$, introduced in the statement of Proposition \ref{prop: Modified Garsia, Rodemich, Rumsey inequality}, as a small fraction of the kinetic time, i.e. $a \ll \alpha^2 $ in the diffusive scaling. 
More precisely, we shall use  \eqref{eq: bound GRR alpha 0} with the parameters
\begin{equation}
\label{eq: choice parameters}
b = 6\, , \quad \gamma = 7/3\, , \quad 
a = (\log\log \mu_\eps)^{-1/10}, \quad
\alpha =  (\log \log\log \mu_\eps)^{-1}.
\end{equation}
We deduce  from (\ref{eq: modulus GRR cutoff}) that, for arbitrary $\delta'>0$,
\begin{align}
\label{eq: borne H -k}
\bbP_\eps \left( \sup_{0 \leq \tau, \sigma \leq {T} \atop | \tau - \sigma | \leq \delta} 
\big\|  \xi^{\eps}_\tau  -  \xi^{\eps}_\sigma \big\|_{-k}^2 \geq \delta' \right) 
\leq &
 \bbP_\eps \left( \sum_j \frac{C^2  B_{a } \big(  \xi^{\eps} (\phi_j) \big)^{1/3}}{\left(1+ |j|^2\right)^k} \delta^{\frac{\gamma-2}{3} }\geq   \frac{\delta'}{4}  \right)\\
& + \bbP_\varepsilon \left( 
\sum_j \frac{4}{\left(1+ |j|^2\right)^k} \sup_{ | \sigma - \tau | \leq 2 a  \atop \sigma,\tau \in [0,T] }
|  \xi^{\eps}_\tau  (\phi_j)-  \xi^{\eps}_\sigma (\phi_j)\big |^2 \geq \frac{\delta'}{4}
\right), \nonumber
\end{align}
where $\phi_j(x) = \exp(2i\pi j\cdot x ) $ are the Fourier modes used to define the norm \eqref{eq: Sobolev -k}.
Since $a \ll \alpha^2$, the two events in the right-hand side of inequality  \eqref{eq: borne H -k} control
different time scales and their probabilities have to be estimated by different methods :
\begin{itemize}
\item for time increments $|\sigma - \tau| \geq a$, by a control on moments using the comparison with the limit process;
\item for small time increments $|\sigma - \tau| \leq 2a$, by reducing to the estimates on the kinetic times obtained in  \cite{BGSS2} (see Proposition 6.2.3). To do this, additional cut-off estimates to control divergences at large velocities are necessary.
\end{itemize}

\bigskip

\medskip

\noindent
{\bf Step 1. Control of the short hydrodynamic increments.}

We are first going to prove that 
\begin{align}
\label{eq: short hydrodynamic times.}
\lim_{\delta \to 0} \lim_{\mu_\eps \to \infty} \bbP_\eps \left( \sum_j \frac{C  B_{a } \big(  \xi^{\eps} (\phi_j) \big)^{1/3}}{\left(1+ |j|^2\right)^k} \delta^{\frac{\gamma-2}{3} }\geq \frac{\delta'}{4} 
\right) = 0 \, .
\end{align}
Assume that the following bound holds 
\begin{equation}
\label{eq: borne B a}
\bbE_\eps \Big( B_{a } \big(  \xi^{\eps}  (\phi)   \big) \Big) \leq C \;  \|\phi\|_{W^{2,\infty}}^6.
\end{equation}
Since for the Fourier basis $ \|\phi_j\|_{W^{2,\infty}} \leq  C |j|^2$, we deduce from 
\eqref{eq: borne B a} that 
for $k > d/2 + 2$, \eqref{eq: short hydrodynamic times.} follows from a Markov inequality 
as $\gamma >2$
\begin{align*}
\bbP_\eps \left( \sum_j \frac{C^2  B_{a } \big( \xi^{\eps} (\phi_j) \big)^{1/3}}{\left(1+ |j|^2\right)^k} \delta^{\frac{\gamma-2}{3} }\geq \frac{\delta'}{4} 
\right) 
\leq  C  \frac{\delta^{\frac{\gamma-2}{3}}}{ \delta'} \sum_j \frac{1}{  \left(1+ |j|^2\right)^k}  \bbE_\eps \left(   B_a \big(  \xi^{\eps} (\phi_j) \big) \right)^{1/3} .
\end{align*}

\medskip

We turn now to the proof of \eqref{eq: borne B a}. As $\gamma = 7/3$, this will be a consequence of the following inequality
\begin{equation}
\label{eq: objectif Ba} 
\forall \tau, \sigma \in [0,T]\, , \qquad 
\bbE_\eps \left[  \Big( \xi^{\eps}_\tau (\phi) - \xi^{\eps}_\sigma (\phi)\Big)^6  \right] 
\indc_{|\tau-\sigma| \geq a }  
\leq C \; \|\phi\|_{W^{2,\infty}}^6 \;  | \tau - \sigma |^{3/2}  \,.
\end{equation}
Applying Lemma \ref{Lem: cut-off approximation}, it is enough to derive \eqref{eq: objectif Ba}  
for the truncated process $\bar\xi^{\eps}$ with cut-off $R = \log \log \log \mu_\eps$ because
\begin{equation*}
\forall \tau \leq T, \qquad 
\bbE_\eps  \left[  \left(  \xi^{\eps}_\tau (\phi) - \bar \xi^{\eps}_\tau  (\phi) \right)^6 \right] 
\leq C \| \phi \|^6_{L^6 (\T^d )} e^{-R/4} 
\leq  C \| \phi \|^6_{L^6 (\T^d )} \,   a^2  \,,
\end{equation*}
with $a$ defined in \eqref{eq: choice parameters}.

Our starting point  is the asymptotic factorization (\ref{WickR}) of the moments leading to the following formula for the time increments  
\begin{equation}
\label{factorization}
\begin{aligned}
 &\left| \bbE_\eps \left[  \Big(\bar\xi^{\eps}_\tau  (\phi)
 - \bar\xi^{\eps}_\sigma (\phi)\Big)^6  \right]  
 - 15 \;
 \bbE_\eps \left[  \Big( \bar\xi^{\eps}_\tau  (\phi)- \bar\xi^{\eps}_\sigma (\phi)\Big)^2\right]^3  \right| \\
&\qquad  
\leq C_6 R^6  \, \| \phi \|_{L^\infty}^6   \,
\Big(\frac {CT} {\alpha^2}\Big) ^{11/2} (\log   \log \mu_\eps )^{-1/4} 
\leq  C \| \phi \|^6_{L^6 (\T^d )} \,  a^2,
\end{aligned}
\end{equation}
uniformly in~$\tau, \sigma \in [0, T]$, with our choice of scaling (\ref{scaling}).

Next we are going to use that, by (\ref{approximate-covarianceR}), the covariance is well approximated by the solution to the linearized Boltzmann equation  (\ref{diffusive-beq}).  
Denoting by $\tilde g_\alpha$ the solution of the linearized Boltzmann equation \eqref{diffusive-beq}  with truncated initial data \eqref{eq: truncated initial data}, we get that 
\begin{equation}
\label{approximate-covariance2}
\begin{aligned}
\sup_{\sigma , \tau \in [0,T ]} & \left| \bbE_\eps \left[  \Big(  \bar\xi^{\eps}_\tau  (\phi)- \bar\xi^{\eps}_\sigma (\phi)\Big)^2\right]  - 2 \int \cM ( \bar  g^0 - \tilde g_\alpha (\tau- \sigma) ) \,   g^0  dxdv\right|\\
&  \leq C R^2 \|\phi\|_{W^{1,\infty}}^2  \, \left( \frac{C T^3  } {\alpha^6} \right)^{1/2}  (\log \log \mu_\eps)^{-1/4} 
+ C\|\phi\|_{L^2}^2  e^{-R/4}  
\leq C \, \|\phi\|_{W^{1,\infty}}^2 \;  a^2  \,,
\end{aligned}
\end{equation}
using the time  invariance of the equilibrium measure
and the control (\ref{cutoff}) to remove  the velocity cutoff on (one of) the initial data $\bar g^0$ in the integral.
>From (\ref{time-control}) we have
$$
\begin{aligned}
& \d_\tau \Big( P  \la \tilde g_\alpha v\ra - \alpha  P\nabla_x  \cdot \la   \tilde g_\alpha \tilde A \ra  \Big) 
-  P\nabla_x \cdot \la v\cdot \nabla_x  ( \tilde g_\alpha \tilde A )\ra=0 \, ,\\
&
\frac{1}{d+2}
 \d_\tau   \Big(   \la \tilde g_\alpha (|v|^2 -d-2)\ra  - 2 \alpha    \nabla_x \cdot  \la \tilde g_\alpha  \tilde B \ra \Big) 
-  \frac{2}{d+2}\nabla_x \cdot \la v\cdot \nabla_x( \tilde g_\alpha  \tilde B)\ra=0,
\end{aligned}
$$
so thanks to the uniform $L^\infty_\tau ( L^2(\cM dxdv))$ bound on $\tilde g_\alpha$ we deduce that 
\begin{equation}
\label{eq: continuite H-2}
\begin{aligned}
& P \la \tilde g_\alpha v\ra -  \alpha  P\nabla_x\cdot \la  \tilde g_\alpha \tilde A \ra \hbox{ is uniformly bounded in } W^{1,\infty}_\tau ( \bbH^{-2} ) ,\\
 &\la \tilde g_\alpha { |v|^2 -(d+2) \over 2}\ra - \alpha \nabla_x \cdot \la  \tilde g_\alpha  \tilde B \ra \hbox{ is uniformly bounded in } W^{1,\infty}_\tau (   \bbH^{-2} ) \,.
 \end{aligned}
\end{equation}
We then have to control the time regularity of the  $O(\alpha)$ terms in \eqref{eq: continuite H-2}. 
>From the uniform $L^\infty_\tau ( L^2(\cM dxdv))$ bound on $\tilde g_\alpha$, we get
that for any polynomial $\mathfrak{p}(v)$ depending only on $v$ 
\begin{equation}
\label{eq: borne H-1}
\begin{aligned}
\forall \tau \in [0,T], \qquad \| \nabla_x  \la  \tilde g_\alpha \mathfrak{p}(v) \ra \|_{ \bbH^{-1} } \leq C.
 \end{aligned}
\end{equation}
Applying the kinetic equation~(\ref{diffusive-beq}), we know that  
$$
\d_\tau \la \tilde g_\alpha \mathfrak{p}(v)  \ra +  \frac{1}{\alpha} \nabla_x  \cdot \la \tilde g_\alpha  \mathfrak{p}(v) v\ra + \frac1{\alpha^2} \la \cL \tilde g_\alpha  \mathfrak{p}(v) \ra =0 
\Rightarrow 
 \left \|   \la  \tilde g_\alpha (\tau)  \mathfrak{p}(v) \ra 
-    \la  \tilde g_\alpha(\sigma) \mathfrak{p}(v)\ra\right \|_{ \bbH^{-1} } 
\leq C {|\tau- \sigma | \over \alpha^2}.
$$
Replacing $\mathfrak{p}$ by $\tilde A , \tilde B$ in the previous estimates,  we conclude that
$$
 \begin{aligned}
& \left \|  \alpha  P\nabla_x\cdot \la  \tilde g_\alpha (\tau)  \tilde A \ra 
- \alpha  P\nabla_x\cdot \la  \tilde g_\alpha(\sigma) \tilde A \ra\right \|_{ \bbH^{-2} } \leq C\min \left(\alpha, {|\tau- \sigma | \over \alpha}\right) \leq C |\tau- \sigma |^{1/2} , \\
&  \left\|  \alpha \nabla_x \cdot \la  \tilde g_\alpha  (\tau) \tilde B \ra
 - \alpha \nabla_x \cdot \la  \tilde g_\alpha(\sigma)   \tilde B \ra\right\|_{\bbH^{-2}} 
\leq C\min \left(\alpha, {|\tau- \sigma | \over \alpha}\right) \leq C |\tau- \sigma |^{1/2}\,.
 \end{aligned}
$$
Therefore,  applying \eqref{eq: continuite H-2}, we deduce that the bulk velocity
$P \la \tilde g_\alpha v\ra$ and temperature $ \la \tilde g_\alpha { |v|^2 -(d+2) \over d+ 2}\ra$ are uniformly bounded in $C^{1/2}_\tau ( \bbH^{-2}_x)$. 
Since the initial data $g^0$  is well prepared (see \eqref{eq: well prepared}), 
we deduce  that the term involving the linearized equation in \eqref{approximate-covariance2} is controlled by 
$$
\begin{aligned}
\Big| \int \cM ( \bar  g^0 - \tilde g_\alpha (\tau- \sigma) ) \,  g^0  dxdv \Big| 
& \leq \left|  \int  \Big( P \la \tilde g_\alpha(\tau - \sigma)  v\ra - P \la \bar g_0 v\ra\Big)  \cdot u_0 dx \right|\\
&  +\frac{d+2}2 \left|  \int \Big( \la \tilde g_\alpha (\tau - \sigma) { |v|^2 -(d+2) \over 2}\ra - \la \bar  g_0 { |v|^2 -(d+2) \over 2}\ra\Big) \theta_0 dx\right|  \\
& \leq C \, \|\phi\|_{W^{2,\infty}}^2 \; | \tau- \sigma  |^{1/2}.
\end{aligned}
$$

\medskip

Combining (\ref{factorization})-(\ref{approximate-covariance2}) and the time regularity of the covariance, we get that for  $|\tau- \sigma | \geq a$
$$
\bbE_\eps \left[  \Big(  \xi^{\eps}_\tau  (\phi) -  \xi^{\eps}_\sigma (\phi)\Big)^6  \right]   
1_{|\tau- \sigma | \geq a}
\leq C \;    \|\phi\|_{W^{2,\infty}}^6 \; | \tau- \sigma  |^{3/2} . 
$$
This completes the proof of Inequality \eqref{eq: objectif Ba}.

\medskip

\noindent
{\bf Step 2. Control of the very short kinetic times.}

Finally, it remains to control the second term in \eqref{eq: borne H -k}.
By splitting the time interval $[0,T]$ into intervals with kinetic time length scale $\alpha^2$, 
the estimate can be reduced, by using the invariant measure and an union bound,  to 
\begin{align}
\label{eq: temps cinetiques}
\bbP_\varepsilon \left( 
\sum_j \frac{4}{\left(1+ |j|^2\right)^k} \sup_{ | \sigma - \tau | \leq 2 a  \atop \sigma,\tau \in [0,T] } |  \xi^{\eps}_\tau  (\phi_j)-  \xi^{\eps}_\sigma (\phi_j)\big |^2 \geq \frac{\delta'}{4}
\right) \leq \frac{T}{\alpha^2} \bbP_\varepsilon \left( \cA \right),
\end{align}
with the notation
\begin{align}
\label{eq: def very short kinetic times}
\cA := \left \{ \sum_j \frac{4}{\left(1+ |j|^2\right)^k} \sup_{ | \sigma - \tau | \leq 2 a  \atop \sigma,\tau \in [0, \alpha^2] } |  \xi^{\eps}_\tau  (\phi_j)-  \xi^{\eps}_\sigma (\phi_j)\big |^2 \geq \frac{\delta'}{4} \right\}.
\end{align}
Recalling that $a \ll \alpha^2$, we are going to show that 
\begin{align}
\label{eq: very short kinetic times}
\lim_{\mu_\eps \to \infty} \frac{1}{\alpha^2} \bbP_\varepsilon \left( \cA \right) = 0,
\end{align}
which is essentially the outcome of Proposition 6.2.3 in  \cite{BGSS2}, however the proof 
cannot be applied directly in our context
and we explain below the necessary adjustments.

\medskip

First of all, the test functions are now unbounded in $v$ (contrary to the Fourier-Hermite modes).
Thus an energy cut-off is necessary. For technical reasons,  we are going to use a larger truncation parameter $\tilde R = (\log \mu_\eps)^2$ instead of 
$R =  \alpha^{-1}$ introduced in \eqref{scaling}. The corresponding truncated process is  defined as in \eqref{u-theta-defR} 
and denoted by $(\tilde \xi^{\eps}_\tau)_{\tau \geq 0}$. 
We are going to check that with high probability 
both processes coincide because all the velocities remain smaller than $\sqrt{\tilde R}$ 
\begin{equation}
\label{eq: truncation tilde}
\lim_{\mu_\eps \to \infty} 
\frac{1}{\alpha^2} \bbP_\varepsilon \left(  \exists i, \quad \sup_{t \leq \alpha} \big| {\bf v}^{\e}_i(t)  \big| >  \sqrt{\tilde R} \right) = 0.
\end{equation}
This can be deduced from a result of  \cite{BGSS3} as follows.
Fix  ${n} = 4d, \eta = \e^{1 - \frac{1}{2d}}$ and 
call {\it microscopic cluster of size~${n}$} a set~$\cG $ of~${n}$ particle configurations in~$\T ^d\times\R^d$ such that~$(z,z') \in \cG\times\cG$ if and only if there 
are~$z_{1} = z,z_2,\dots,z_{\ell}= z'$ in~$\cG$ such that
\begin{equation*}
\label{eq: cut off delta}
 |x_{i}-x_{i+1} | \leq 3   \, \sqrt{\tilde R} \, \eta , \quad \forall1 \leq i \leq \ell-1 \, .
\end{equation*} 
%\begin{equation}
%\label{eq: choix parametres}
%\eps \ll \tau \ll 1 \ll \theta \ll \big(\log |\log \eps|\big)^{1/4} \quad \mbox{and} \quad {n} = 4d\;, \quad  {\mathbb V} =|\log  \eps|\;, \, \quad \delta = \e^{1 - \frac{1}{2d}}\;.
%\end{equation}
Let $\Upsilon^\eps_N$ be the set of initial configurations~${\bf Z}^{\eps 0}_N \in {\mathcal D}^{\eps}_{N}$ such that for any integer~$1 \leq k \leq \frac{\alpha}{ \eta}$, the configuration at time~$k \eta$ satisfies
\begin{equation}
\label{eq: contrainte vitesses}
\forall 1 \leq j \leq N , \qquad    |v_j | \leq \frac{\sqrt{\tilde R}}{{n}} \, , 
\end{equation}
and any microscopic cluster of particles is of size at most~$ {n}$.
Adapting to our framework the proof of Proposition 2.7 of \cite{BGSS3} implies that 
\begin{equation}
\label{eq: conditioning}
\bbP_\eps \big(  ^c \Upsilon^\eps_\cN \big) \leq  \frac{1}{\alpha^n}  \; \eps^d.
\end{equation}
We check that for any configuration in $\Upsilon^\eps_\cN$, the velocities are bounded from above by $\sqrt{\tilde R}$  during the kinetic time interval $[0, \alpha]$. 
Indeed, at each intermediate time $k \eta$, the velocities of  configurations in $\Upsilon^\eps_\cN$,  are    smaller than $\frac{\sqrt{\tilde R}}{n}$  by \eqref{eq: contrainte vitesses}. Furthermore the   clusters  are all 
of size less than ${n}$ and in the time interval $[k \, \eta, (k+1) \eta]$ particles within a cluster cannot interact with particles in other clusters. As the total kinetic energy of a finite number of particles is preserved by the hard sphere dynamics, the velocity of each particle will remain less than $\sqrt{\tilde R}$.
Thus \eqref{eq: truncation tilde} is implied by \eqref{eq: conditioning}.

\medskip

We are now in position to complete the proof of \eqref{eq: very short kinetic times}.
Thanks to \eqref{eq: truncation tilde}, it is enough to replace the event $\cA$ by the 
similar event $\tilde \cA$ for the process $(\tilde \xi^{\eps}_\tau)_{\tau \geq 0}$.
It thus remains to prove 
\begin{equation}
\label{eq: very short kinetic times tilde}
\lim_{\mu_\eps \to \infty} \frac{1}{\alpha^2} \bbP_\varepsilon \left( \tilde \cA \right) = 0.
\end{equation}
The statement of Proposition 6.2.3 from \cite{BGSS2} is not precise enough to conclude directly mainly due to the diverging prefactor $\frac{1}{\alpha^2}$. However all the required estimates can be found in \cite{BGSS2} and we are going to detail the relevant parts of the argument.

We proceed as in \eqref{eq: borne H -k} and introduce an additional time cut-off $\mu_\eps^{-7/3}$ instead of $a$ to filter the very small scales
\begin{align*}
\frac{1}{\alpha^2} \bbP_\varepsilon \left( \tilde \cA \right)  = 
\frac{1}{\alpha^2} & \bbP_\eps \left( \sup_{0 \leq \tau, \sigma \leq \alpha^2 \atop | \tau - \sigma | \leq 2a} 
\big\| \tilde \xi^{\eps}_\tau  - \tilde  \xi^{\eps}_\sigma \big\|_{-k}^2 \geq \frac{\delta'}{16}  \right) 
\leq 
\frac{1}{\alpha^2} \bbP_\eps \left( \sum_j \frac{C^2  \hat B_{\mu_\eps^{-7/3}} \big(  \xi^{\eps} (\phi_j) \big)^{1/3}}{\left(1+ |j|^2\right)^k} a^{\frac{2 \gamma-4}{3} }\geq   \frac{\delta'}{64}  \right) \\
 & \quad 
 + \frac{1}{\alpha^2} \bbP_\varepsilon \left( 
\sum_j \frac{4}{\left(1+ |j|^2\right)^k} \sup_{ | \sigma - \tau | \leq 2 \mu_\eps^{-7/3}  \atop \sigma,\tau \in [0,\alpha^2] }
| \tilde \xi^{\eps}_\tau  (\phi_j) - \tilde \xi^{\eps}_\sigma (\phi_j)\big |^2 \geq \frac{\delta'}{64}
\right), \nonumber
\end{align*}
with the analogous notation of  \eqref{eq: bound GRR alpha 0} on this short time scale
\begin{equation*}
\hat B_{\mu_\eps^{-7/3}} (F) : = 
\int_0^{\alpha^2}  \int_0^{\alpha^2}  d \sigma d \tau \;
 \frac{| F_\tau - F_\sigma |^b}{| \tau - \sigma |^\gamma} {\bf 1}_{| \tau - \sigma | > \mu_\eps^{-7/3}}\quad
\text{with} \quad b = 6\, ,  \gamma = 7/3.
\end{equation*}
In our procedure, it was necessary to use first a time cut-off $a$ in  \eqref{eq: borne H -k}  in order to reduce to estimates in the kinetic time scale. Indeed the error term \eqref{approximate-covariance2}
occurring in the comparison with the limiting equations on the diffusive time scale $[0,T]$ was too crude to be efficient up to the smallest time scale $\mu_\eps^{-7/3}$. On the kinetic scale better controls can be derived and one can show as in Lemma 6.2.6 of \cite{BGSS2} (with the Remark 6.2.8 to take care of the large velocities) that 
\begin{align*}
\frac{1}{\alpha^2} \bbP_\eps \left( \sum_j \frac{C^2  \hat B_{\mu_\eps^{-7/3}} \big( \tilde \xi^{\eps} (\phi_j) \big)^{1/3}}{\left(1+ |j|^2\right)^k} a^{\frac{2 \gamma-4}{3} }
\geq   \frac{\delta'}{64}  \right)
\leq C^2 \frac{64}{\alpha^2 \, \delta'} a^{\frac{2 \gamma-4}{3} }.
 \end{align*}
As $a \ll \alpha$, this term vanishes in the diffusive limit.
By using the proof of Lemma 6.2.5 of \cite{BGSS2} (with the Remark 6.2.8 to take care of the logarithmic divergence), we deduce that second term vanishes also in the diffusive limit
\begin{align*}
\frac{1}{\alpha^2} \bbP_\varepsilon \left( 
\sum_j \frac{4}{\left(1+ |j|^2\right)^k} \sup_{ | \sigma - \tau | \leq 2 \mu_\eps^{-7/3}  \atop \sigma,\tau \in [0,\alpha^2] }
| \tilde \xi^{\eps}_\tau  (\phi_j) - \tilde \xi^{\eps}_\sigma (\phi_j)\big |^2 \geq \frac{\delta'}{64}
\right) \leq \frac{C}{\alpha^2}   \mu_\eps^{-1/3} \to 0.
\end{align*}
Combining the previous results, \eqref{eq: very short kinetic times tilde} holds.
This completes the proof of Proposition \ref{thmtight diffusive}.
\end{proof}

\end{document}